\documentclass[12pt]{article}
\usepackage{amsmath,amsthm,amssymb, color}
\input amssym.def
\input amssym.tex
\date{}
\newtheorem{lemma}{Lemma}
\newtheorem{proposition}{Proposition}
\newtheorem{theorem}{THEOREM}
\newtheorem{corollary}{Corollary}
\newtheorem{remark}{Remark}

\begin{document}

\title{{\bf Rademacher functions in Morrey spaces}}

\author {Sergei V. Astashkin
{\small and} Lech Maligranda}

\date{}

\maketitle

\renewcommand{\thefootnote}{\fnsymbol{footnote}}


\vspace{-5mm}

\begin{abstract}
\noindent {\footnotesize The Rademacher sums are investigated in the Morrey spaces $M_{p, w}$ on $[0, 1]$ for 
$1 \leq p < \infty$ and weight $w$ being a quasi-concave function. They span $l_2$ space in $M_{p, w}$ if and 
only if the weight $w$ is smaller than $\log_2^{-1/2} \frac{2}{t}$ on $(0, 1)$. Moreover, if $1 < p < \infty$ the 
Rademacher sunspace ${\cal R}_p$ is complemented in $M_{p, w}$ if and only if it is isomorphic to $l_2$. 
However, the Rademacher subspace ${\cal R}_1$ is not complemented in $M_{1, w}$ for any quasi-concave 
weight $w$. In the last part of the paper geometric structure of Rademacher subspaces in Morrey spaces $M_{p, w}$ is described. 
It turns out that for any infinite-dimensional subspace $X$ of ${\cal R}_p$ the following alternative holds: either $X$ 
is isomorphic to $l_2$ or $X$ contains a subspace which is isomorphic to $c_0$ and is complemented in ${\cal R}_p$.}
\end{abstract}

\section{Introduction and preliminaries}

The well-known Morrey spaces introduced by Morrey in 1938 \cite{Mo38} in relation to the study of partial differential 
equations were widely investigated during last decades, including the study of classical operators of harmonic 
analysis: maximal, singular and potential operators -- in various generalizations of these spaces. In the theory of partial 
differential equations, along with the weighted Lebesgue spaces, Morrey-type spaces also play an important role. 
They appeared to be quite useful in the study of the local behavior of the solutions of partial differential equations, a priori 
estimates and other topics.

Let $0 < p < \infty$, $w$ be a non-negative non-decreasing function on $[0, \infty),$ and $\Omega$ a domain in $\mathbb R^n$. 
The {\it Morrey space} $M_{p, w} = M_{p, w}(\Omega)$ is the class of Lebesgue measurable real functions $f$ on $\Omega$ 
such that 
\begin{equation} \label{1}
\|f\|_{M_{p, w}} = \sup_{0 < r < {\rm diam}(\Omega), \, x_0 \in \Omega} ~ w(r) \left( \frac{1}{r} \int_{B_r(x_0) \cap \Omega} | f(t) |^p 
\,dt \right)^{1/p} < \infty,
\end{equation}
where $B_r(x_0)$ is a ball with the center at $x_0$ and radius $r$. It is a quasi-Banach ideal space on $\Omega$. 
The so-called ideal property means that if $|f| \leq |g|$ a.e. on $\Omega$ and $g \in M_{p, w}$, then $ f\in M_{p, w}$ and 
$\|f\|_{M_{p, w}} \leq \|g \|_{M_{p, w}}$. In particular, if $w(r) = 1$ then $M_{p, w}(\Omega) = L_{\infty}(\Omega)$, if $w(r) = r^{1/p}$ 
then $M_{p, w}(\Omega) = L_p(\Omega)$ and in the case when $w(r) = r^{1/q}$ with $0 < p \leq q <\infty$ $M_{p,w}(\Omega)$ are the
classical Morrey spaces, denoted shortly by $M_{p, q}(\Omega)$ (see \cite[Part 4.3]{KJF77}, \cite{LR13}, \cite{Pe69} and \cite{Zo86}). 
Moreover, as a consequence of the H\"older-Rogers inequality we obtain monotonicity with respect to $p$, that is,
$$
M_{p_1, w}(\Omega) \stackrel {1}\hookrightarrow M_{p_0, w}(\Omega) ~~ {\rm if} ~~ 0 < p_0 \leq p_1 < \infty.
$$
For two quasi-Banach spaces $X$ and $Y$ the symbol $X \stackrel {C}\hookrightarrow Y$ means that the embedding 
$X \subset Y$ is continuous  and $\|f\|_{Y} \leq C \|f\|_{X}$ for all $f \in X$.

It is easy to see that in the case when $\Omega = [0, 1]$ quasi-norm (\ref{1}) can be defined as follows
\begin{equation} \label{2}
\|f\|_{M_{p, w}} = \sup_{I} ~ w(|I|) \left( \frac{1}{|I|}\int_I | f(t) |^p\,dt \right)^{1/p},
\end{equation}
where the supremum is taken over all intervals $I$ in $[0, 1]$.  In what follows $| E|$ is the Lebesgue measure of a set $E \subset \mathbb R$.

The main purpose of this paper is the investigation of the behaviour of Rademacher sums
$$
R_n(t) = \sum_{k=1}^n a_k r_k(t), ~ a_k \in {\Bbb R} ~{\rm for} ~k = 1, 2, ..., n, ~{\rm and} ~ n \in {\Bbb N}
$$
in general Morrey spaces $M_{p, w}$. Recall that the Rademacher functions on $[0, 1]$ are defined by 
$r_k(t) = {\rm sign}(\sin 2^k \pi t), ~ k \in {\Bbb N}, t \in [0, 1]$. 

The most important tool in studying Rademacher sums in the classical $L_p$-spaces and in general rearrangement invariant 
spaces is the so-called {\it Khintchine inequality} (cf. \cite[p. 10]{DJT}, \cite[p. 133]{AK06}, \cite[p. 66]{LT77} and \cite[p. 743]{As09}): 
if $0 < p < \infty$, then there exist constants $A_p, B_p > 0$ such that for any sequence of real numbers $\{a_k\}_{k=1}^n$ and any 
$n \in {\Bbb N}$ we have
\begin{equation}\label{3}
A_p \Big(\sum_{k=1}^n |a_k|^2 \Big)^{1/2} \leq \| R_n \|_{L_p[0, 1]} \leq B_p \Big(\sum_{k=1}^n |a_k |^2\Big)^{1/2}.
\end{equation}
Therefore, for any $1 \leq p < \infty$, the Rademacher functions span in $L_p$ an isomorphic copy of $l_2$. 
Also, the subspace $[r_n] $ is complemented in $L_p$ for $1 < p < \infty$ and is not 
complemented in $L_1$ since no complemented infinite dimensional subspace of $L_1$ can be reflexive. In $L_{\infty}$, the 
Rademacher functions span an isometric copy of $l_1$, which is uncomplemented.

The only non-trivial estimate for Rademacher sums in a general rearrangement invariant (r.i.) space $X$ on $[0,1]$ 
is the inequality
\begin{equation}\label{4}
\| R_n \|_{X}  \leq C \, \Big(\sum_{k=1}^{n} |a_k|^2 \Big)^{1/2},
\end{equation}
where a constant $C> 0$ depends only on $X$. The reverse inequality to (\ref{4}) is always true because $X \subset L_1$ and we 
can apply the left-hand side inequality from (\ref{3}) for $L_1$. Paley and Zygmund \cite{PZ30} proved already in 1930 that 
estimate (\ref{4}) holds for $X=G$, where $G$ is the closure of $L_{\infty}[0, 1]$ in the Orlicz space $L_M[0, 1]$ generated by the 
function $M(u) = e^{u^2} -1$. The proof can be found in Zygmund's classical books \cite[p. 134] {Zy35} and \cite[p. 214]{Zy59}. 

Later on Rodin and Semenov \cite{RS75} showed that estimate (\ref{4}) holds if and only if $G \subset X$. This inclusion means 
that $X$ in a certain sense ``lies far'' from $L_{\infty}[0,1]$. In particular, $G$ is contained in every $L_p[0,1]$ for $p < \infty$.
Moreover, Rodin-Semenov \cite{RS79} and Lindenstrauss-Tzafriri \cite[pp. 134-138]{LT} proved that  $[r_n]$ is complemented in 
$X$ if and only if $G \subset X \subset G^{\prime}$, where $G^{\prime}$ denotes the K\"othe dual space to $G$.

In contrast, Astashkin  \cite{As01} studied the Rademacher sums in r.i. spaces which are situated very ``close" to $L_{\infty}$. 
In such a case a rather precise description of their behaviour may be obtained by using the real method of interpolation 
(cf. \cite{BK91}). Namely, every space $X$ that is  interpolation between the spaces $L_{\infty}$ and $G$ can be represented 
in the form $X = (L_\infty,G)_\Phi^K$, for some parameter $\Phi$ of the real interpolation method, and then 
$\|\sum_{k=1}^{\infty} a_k r_k \|_X \approx \| \{a_k\}_{k=1}^{\infty} \|_F$, where $F=(l_1,l_2)_\Phi^K$.

Investigations of Rademacher sums in r.i. spaces are well presented in the books by Lindenstrauss-Tzafriri \cite{LT}, 
Krein-Petunin-Semenov \cite{KPS} and Astashkin \cite{As09}. At the same time, a very few papers are devoted to considering 
Rademacher functions in Banach function spaces, which are not r.i. Recently, Astashkin-Maligranda \cite{AM10} initiated studying 
the behaviour of Rademacher sums in a weighted Korenblyum-Kre\u\i n-Levin space $K_{p,w}$, for $0 < p < \infty$ and  
a quasi-concave function $w$ on $[0, 1]$, equipped with the quasi-norm
\begin{equation}\label{5}
\|f\|_{K_{p, w}} = \sup_{0 < x \leq 1} w(x) \, \left( \frac{1}{x}\int_{0}^{x} | f(t) |^p\,dt \right)^{1/p}
\end{equation}
(cf. \cite{KKL}, \cite{LZ66}, \cite[pp. 469-470]{Za83}, where $w(x) = 1$). If the supremum in (\ref{2}) is taken over all subsets of 
$[0, 1]$ of measure $x$, then we obtain an r.i. counterpart of the spaces $M_{p, w}$ and $K_{p, w}$, the Marcinkiewicz space 
$M_{p, w}^{(*)}[0, 1],$ with the quasi-norm
\begin{equation} \label{6}
\|f\|_{M_{p, w}^{(*)}} = \sup_{0 < x \leq 1} ~ w(x) \left( \frac{1}{x}\int_0^x f^*(t)^p\,dt \right)^{1/p},
\end{equation}
where $f^*$ denotes the non-increasing rearrangement of $| f |$.

In what follows we consider only function spaces on $[0, 1]$. Therefore, the weight $w$ will be a non-negative non-decreasing 
function on $[0, 1]$ and without loss of generality we will assume in the rest of the paper that $w(1) = 1$. Then, we have
\begin{equation}  \label{7}
L_{\infty} \stackrel {1} \hookrightarrow M_{p, w}^{(*)} \stackrel {1}\hookrightarrow  M_{p, w} \stackrel {1}\hookrightarrow K_{p, w} \stackrel {1} \hookrightarrow L_p
\end{equation} 
because the corresponding suprema in (\ref{5}), (\ref{2}) and (\ref{6}) are taken over larger classes of subsets of $[0, 1]$. 

Observe that if $\lim_{t \rightarrow 0^+} w(t) > 0$, then $M_{p, w} = M_{p, w}^{(*)} = L_{\infty}$, and if 
$\sup_{0 < t \leq 1} w(t) \, t^{-1/p}  < \infty$, then $ M_{p, w} = L_p$ with equivalent quasi-norms. However, under appropriate 
assumptions on a weight $w$ the second and the third inclusions in (\ref{7}) are proper.

\begin{proposition} \label{Pro1} 
\begin{itemize} 
\item[(i)] If $\lim_{t \rightarrow 0^+} w(t) \, t^{-1/p} = \infty$, then there exists $f \in K_{p, w} \setminus M_{p, w}$.
\item[(ii)] If $w(t) \, t^{-1/p}$ is a non-increasing function on $(0, 1]$ and $\lim_{t \rightarrow 0^+} w(t) = 
\lim_{t \rightarrow 0^+} \frac{t^{1/p}}{w(t)} = 0$, then there exists $g \in M_{p, w} \setminus M_{p, w}^{(*)}$.
\end{itemize}
\end{proposition}

\begin{proof}
(i) Since $\lim_{t \rightarrow 0^+} w(t) t^{-1/p} = \infty$, there exists a sequence $\{t_k\} \subset (0, 1]$ such that 
$t_k \searrow 0, t_1 \leq 1/2$ and $w(t_k) t_k^{-1/p} \nearrow \infty$. Let us denote $v(t) = w(t) \, t^{-1/p}$ and 
$$
g(s): = \sum_{k=1}^{\infty} \Big( v(t_k)^{-p/2} - v(t_{k+1})^{-p/2} \Big)^{1/p} (t_k - t_{k+1})^{-1/p} \chi_{(t_{k+1}, t_k]}(s).
$$
Note that, by definition, ${\rm supp} \, g \subset [0, 1/2]$. Then, for every $k \in \mathbb N$
\begin{eqnarray*}
\int_0^{t_k} |g(s)|^p \, ds
&=&
\sum_{i=k}^{\infty} \int_{t_{i+1}}^{t_i} g(s)^p \, ds \\
&=& 
\sum_{i=k}^{\infty} \frac{v(t_i)^{-p/2} - v(t_{i+1})^{-p/2}}{t_i - t_{i+1}} (t_i - t_{i+1}) = v(t_k)^{-p/2}.
\end{eqnarray*}
In particular, we see that $g \in L_p$. Let $f(t):= g(t + \frac{1}{2})$ for $0 \leq t \leq 1$. Then $\| f \|_p = \| g \|_p$, and therefore
$f \in L_p$. Moreover, since ${\rm supp} f \subset [1/2, 1]$, we obtain $f \in K_{p, w}$. In fact, 
\begin{eqnarray*}
\| f \|_{K_{p, w}} 
&=& 
\sup_{0 < x \leq 1} w(x) \left( \frac{1}{x} \int_0^x |f(t)|^p \, dt \right)^{1/p} =
\sup_{\frac{1}{2} \leq x \leq 1} \frac{w(x)}{x^{1/p}} \Big( \int_{1/2}^x |f(t)|^p \, dt \Big)^{1/p} \\
&\approx&
\sup_{\frac{1}{2} \leq x \leq 1} \Big( \int_{1/2}^x |f(t)|^p \, dt \Big)^{1/p} = \| f \|_{L_p} < \infty.
\end{eqnarray*}
At the same time, if $I_k:= [\frac{1}{2}, t_k + \frac{1}{2}], k = 1, 2, \ldots$, we have
\begin{equation*}
w(|I_k|) \left( \frac{1}{|I_k|} \int_{I_k} |f(t)|^p \, dt \right)^{1/p} = v(t_k) \Big ( \int_0^{t_k} |g(s)|^p \, ds \Big)^{1/p} 
= v(t_k) \cdot v(t_k)^{-1/2} = v(t_k)^{1/2}.
\end{equation*}
Since $v(t_k)  \nearrow \infty$ as $k \rightarrow \infty$, we conclude that $f \not \in M_{p, w}$.

(ii) It is easy to find a function $g \in L_p \setminus M_{p, w}^{(*)}$. Next, by the main result of the paper \cite{AA81},  
there exist a function $f \in M_{p, w}$ and constants $c_0 > 0$ and $\lambda_0 > 0$ such that
$$
\Big| \{t \in [0, 1]: |f(t)| > \lambda \} \Big| \geq c \, \Big| \{t \in [0, 1]: |g(t)| > \lambda \} \Big|
$$
for all $\lambda \geq \lambda_0$. Clearly, since $g \not \in M_{p, w}^{(*)}$, from the last inequality it follows
that $f \not \in M_{p, w}^{(*)}$.
\end{proof}

The proof of Proposition \ref{Pro1} (ii) shows also that the Morrey space $M_{p, w}$ is not an r.i. space whenever 
$w(t) \, t^{-1/p}$ is a non-increasing function on $(0, 1]$ and $\lim_{t \rightarrow 0^+} w(t) = \lim_{t \rightarrow 0^+} \frac{t^{1/p}}{w(t)} = 0$.

For a normed ideal space $X = (X, \|\cdot\|)$ on $[0, 1]$ the {\it K{\"o}the dual} (or {\it associated space}) $X^{\prime}$ 
is the space of all real-valued Lebesgue measurable functions defined on $[0, 1]$ such that the {\it associated norm}
$$
\|f\|_{X^{\prime}}: = \sup_{g \in X, ~\|g\|_{X} \leq 1} \int_0^1 |f(x) g(x) | \, dx
$$
is finite. The K{\"o}the dual $X^{\prime}$ is a Banach ideal space. Moreover, $X  \stackrel {1}\hookrightarrow X^{\prime \prime}$ 
and we have equality $X = X^{\prime \prime}$ with $\|f\| = \|f\|_{X^{\prime \prime}}$ if and only if the norm in $X$ has the 
{\it Fatou property}, that is, if $0 \leq f_{n} \nearrow f$ a.e. on $[0, 1]$ and $\sup_{n \in {\bf N}} \|f_{n}\| < \infty$, then $f \in X$ 
and $\|f_{n}\| \nearrow \|f\|$.

Denote by $\cal D$ the set of all dyadic intervals $I_k^n = [(k-1) \, 2^{-n}, k \, 2^{-n}]$, where $n = 0, 1, 2, \ldots$
and $ k = 1, 2, \ldots, 2^n.$ If $f$ and $g$ are nonnegative functions (or quasi-norms), then the symbol $f \approx g$  means 
that $C^{-1}\, g \leq f \leq C\, g$ for some $C \geq 1$.
Moreover, we write $X \simeq Y$ if Banach spaces $X$ and $Y$ are isomorphic. 

The paper is organized as follows. After Introduction, in Section 2 the behaviour of Rademacher sums in Morrey spaces 
 is described (see Theorem 1). The main result of Section 3 is Theorem 2, which states that the Rademacher subspace
 ${\cal R}_p, 1 < p < \infty$, is complemented in the Morrey space $M_{p, w}$  if and only if ${\cal R}_p$ is isomorphic to $l_2$ 
 or equivalently if $\sup_{0 < t \leq 1} w(t) \log_2^{1/2} (2/t) < \infty$.
In the case when $p = 1$ situation is different, which is the contents of Section 4, where we are proving in Theorem 3 that the 
subspace ${\cal R}_1$ is not complemented in $M_{1, w}$ for any quasi-concave weight $w$.
Finally, in Section 5, the geometric structure of Rademacher subspaces in Morrey spaces is investigated (see Theorem 4).

\section{ Rademacher sums in Morrey spaces}

We start with the description of behaviour of Rademacher sums in the Morrey spaces $M_{p, w}$ defined by quasi-norms (\ref{2}), 
where $0 < p < \infty$ and $w$ is a non-decreasing function on $[0, 1]$ satisfying the doubling condition $w(2t) \leq C_0 \, w(t)$ for 
all $t \in (0, 1/2]$ with a certain $C_0 \geq 1$.
 
\begin{theorem} \label{Thm1} With constants depending only on $p$ and $w$
\begin{equation} \label{8}
\Big\| \sum_{k=1}^{\infty} a_k \, r_k  \Big\|_{M_{p, w}} \approx \| \{a_k\}_{k=1}^{\infty} \|_{l_2} + \sup_{m \in \mathbb N} \Big( w(2^{-m}) 
\sum_{k=1}^m | a_k | \Big). 
\end{equation}
\end{theorem}

\begin{proof} Firstly, let $ 1 \leq p < \infty$. Consider an arbitrary interval $I \in {\cal D}$, i.e., $I = I_k^m,$ with $m \in \mathbb N$ and 
$k = 1, 2, \ldots, 2^m$. Then, for every $f = \sum_{k=1}^{\infty} a_k \, r_k$, we have
$$
\Big(\int_I | f(t)|^p \, dt \Big)^{1/p} = \Big( \int_I \Big|\sum_{k=1}^m a_k \varepsilon_k + \sum_{k=m+1}^{\infty} a_k r_k(t) \Big|^p \, dt \Big)^{1/p},
$$
where $\varepsilon_k = {\rm sign} \, {r_k}_{\big| I}, k = 1, 2, \ldots, m$. Since the functions
$$
\sum_{k=1}^m a_k \varepsilon_k +\sum_{k=m+1}^{\infty} a_k r_k(t) ~~ {\rm and} 
~~ \sum_{k=1}^m a_k \varepsilon_k -\sum_{k=m+1}^{\infty} a_k r_k(t)
$$
are equimeasurable on the interval $I$, it follows that
\begin{eqnarray*}
\Big (\int_I | f(t)|^p \, dt \Big)^{1/p} 
&=& 
\frac{1}{2} \Big ( \int_I \Big|\sum_{k=1}^m a_k \varepsilon_k +\sum_{k=m+1}^{\infty} a_k r_k(t) \Big|^p \, dt \Big)^{1/p} \\
&+& 
\frac{1}{2} \Big ( \int_I \Big|\sum_{k=1}^m a_k \varepsilon_k -\sum_{k=m+1}^{\infty} a_k r_k(t) \Big|^p \, dt \Big)^{1/p},
\end{eqnarray*}
whence by the Minkowski triangle inequality we obtain
\begin{equation*}
\Big (\int_I | f(t)|^p \, dt \Big)^{1/p} \geq \Big ( \int_I \Big|\sum_{k=1}^m a_k \varepsilon_k \Big|^p \, dt \Big)^{1/p} \\
= 2^{-m/p}\, \Big|\sum_{k=1}^m a_k \varepsilon_k \Big|
\end{equation*}
for every $m = 1, 2, \ldots$. Clearly, one may find $i = 1, 2, \ldots, 2^m$ such that 
${r_k}_{\big | I_i^m} = {\rm sign} \, a_k$,  for all $k = 1, 2, \ldots, m$. Therefore, for every $m = 1, 2, \ldots$
$$
\left(\frac{1}{|I|} \int_I | f(t)|^p \, dt \right)^{1/p} \geq \sum_{k=1}^m |a_k|,
$$
and so
\begin{equation*}
\| f \|_{M_{p, w}} \geq \sup_{m \in \mathbb N} w(2^{-m}) \sum_{k=1}^m |a_k|.
\end{equation*}
On the other hand, by (\ref{7}) and (\ref{3}) we have
\begin{equation*}
\| f \|_{M_{p, w}} \geq \| f \|_{L_p} \geq A_p \, \| \{a_k\}_{k=1}^{\infty} \|_{l_2}.
\end{equation*}
Combining these inequalities, we obtain
\begin{equation*}
\| f \|_{M_{p, w}} \geq \frac{A_p}{2} \Big ( \| \{a_k\}_{k=1}^{\infty} \|_{l_2} +  \sup_{m \in \mathbb N} w(2^{-m}) \sum_{k=1}^m |a_k| \Big).
\end{equation*}

Let us prove the reverse inequality. For a given interval $I \subset [0, 1]$ we can find two adjacent dyadic intervals $I_1$ and $I_2$ 
of the same length such that
\begin{equation} \label{9}
I \subset I_1 \cup I_2 ~~ {\rm and} ~~ \frac{1}{2}\, |I_1| \leq |I| \leq 2\, |I_1|.
\end{equation}
If $| I_1| = | I_2| = 2^{-m}$, then by the Minkowski triangle inequality and inequality in (\ref{3}) we have
\begin{eqnarray*}
\Big (\int_{I_1} | f (t)|^p \,dt \Big)^{1/p} 
&=&
\Big ( \int_{I_1} \Big|\sum_{k=1}^m a_k \varepsilon_k + \sum_{k=m+1}^{\infty} a_k r_k(t) \Big|^p \, dt \Big )^{1/p}\\
&\leq&
\Big( \int_{I_1} \Big|\sum_{k=1}^m a_k \varepsilon_k \Big|^p \, dt \Big )^{1/p} + \Big ( \int_{I_1} \Big|\sum_{k=m+1}^{\infty} a_k r_k(t) \Big|^p \, dt \Big )^{1/p}\\
&\leq&
2^{-m/p} \, \sum_{k=1}^m | a_k| + 2^{-m/p} \, \Big ( \int_0^1 \Big| \sum_{k=m+1}^{\infty} a_k r_{k-m}(t) \Big|^p \, dt \Big )^{1/p}\\
&\leq&
2^{-m/p} \, \sum_{k=1}^m | a_k| + 2^{-m/p} \, B_p \,  \| \{a_k\}_{k=1}^{\infty} \|_{l_2}.
\end{eqnarray*}
The same estimate holds also for the integral $\big (\int_{I_2} | f (t)|^p \,dt \big)^{1/p}$. Therefore, by (\ref{9}),
\begin{eqnarray*}
\Big (\frac{1}{| I |} \int_{I} | f (t)|^p \,dt \Big)^{1/p} 
&\leq& 
2^{1/p} \, \Big (\frac{1}{| I_1 |} \int_{I_1} | f (t)|^p \,dt + \frac{1}{| I_2 |} \int_{I_2} | f (t)|^p \,dt \Big)^{1/p}\\
&\leq&
4^{1/p} \, B_p \, \Big( \sum_{k=1}^m | a_k| + \| \{a_k\}_{k=1}^{\infty} \|_{l_2} \Big)
\end{eqnarray*}
and
\begin{eqnarray*}
w(| I |) \, \Big (\frac{1}{| I |} \int_{I} | f (t)|^p \,dt \Big)^{1/p} 
&\leq&
w(2 \cdot 2^{-m}) \, 4^{1/p} \, B_p \, \Big( \sum_{k=1}^m | a_k| + \| \{a_k\}_{k=1}^{\infty} \|_{l_2} \Big)\\
&\leq&
C_0 \cdot 4^{1/p} \, B_p \, w(2^{-m})  \, \Big( \sum_{k=1}^m | a_k| + \| \{a_k\}_{k=1}^{\infty} \|_{l_2} \Big).
\end{eqnarray*}
Hence, using definition of the norm in $M_{p, w}$, we obtain
\begin{equation*}
\| f \|_{M_{p, w}} \leq
C_0 \cdot 4^{1/p} \, B_p \, \Big( \sup_{m \in \mathbb N} w(2^{-m}) \, \sum_{k=1}^m | a_k| + \| \{a_k\}_{k=1}^{\infty} \|_{l_2} \Big).
\end{equation*}
The same proof works also in the case when $0 < p < 1$ with the only change that the $L_p$-triangle inequality contains 
constant $2^{1/p-1}$. 
\end{proof}

In the rest of the paper, a weight function $w$ is assumed to be {\it quasi-concave on $[0, 1]$}, that is, $w(0) = 0, w$ is non-decreasing, 
and ${w(t)}/{t}$ is non-increasing on $(0, 1]$. Moreover, as above, we assume that $w(1) = 1$.

Recall that a basic sequence $\{x_k\}$ in a Banach space $X$ is called {\it subsymmetric} if it is unconditional and is equivalent in $X$ to any its subsequence.

\begin{corollary} \label{Cor1}
For every $1 \leq p < \infty$ $\{r_k\}$ is an unconditional and not subsymmetric basic sequence in $M_{p, w}$.
\end{corollary}

\begin{corollary} \label{Cor2}
Let $1 \leq p < \infty$. The Rademacher functions span $l_2$ space in $M_{p, w}$ if and only if 
\begin{equation} \label{10}
 \sup_{0 < t \leq 1} w(t) \log_2^{1/2} (2/t) < \infty.
\end{equation} 
\end{corollary}

\begin{proof} If (\ref{10}) holds, then for all $m \in \mathbb N$ we have $w(2^{-m}) \, m^{1/2} \leq C$. Using the 
H\"older-Rogers inequality, we obtain
$$
w(2^{-m}) \, \sum_{k=1}^m |a_k| \leq w(2^{-m})\Big(\sum_{k=1}^m |a_k|^2\Big)^{1/2} m^{1/2} \le C \, \Big(\sum_{k=1}^m |a_k|^2\Big)^{1/2}.
$$
Therefore, from (\ref{8}) it follows that $\| \sum_{k=1}^{\infty} a_k \, r_k \|_{M_{p, w}} \approx \| \{a_k\}\|_{l_2}$. 

Conversely, suppose that condition (\ref{10}) does not hold. Then, by the quasi-concavity of $w$, there exists a sequence 
of natural numbers $m_k \rightarrow \infty$ such that
\begin{equation}\label{11}
w(2^{-m_k}) \, m_k^{1/2}  \rightarrow \infty ~{\rm as} ~ k  \rightarrow \infty.
\end{equation}
Consider the Rademacher sums $R_k(t) = \sum_{i=1}^k a_i^k \, r_i(t)$ corresponding to the sequences
of coefficients $a^k = (a_i^k)_{i=1}^{m_k}$, where $a_i^k = m_k^{-1/2}, ~1 \leq i \leq m_k$. We have $\| a^k\|_{l_2} = 1$ for all
$k = 1, 2, ... $ However, $\sum_{i=1}^{m_k} a_i^{k} = m_k^{1/2} ~(k = 1, 2, ...)$, which together with (\ref{11}) and (\ref{8}) 
imply that $\| R_k\|_{M_{p, w}} \rightarrow \infty$ as $k \rightarrow \infty$.
\end{proof}

\begin{remark} \label{rem1}
{\rm The Rademacher functions span $l_2$ in each of the spaces $M_{p, w}^{(*)}, M_{p, w}$ and $K_{p, w},$ $1 \leq p < \infty$
(see embeddings (\ref{7})). In fact, the Orlicz space $L_M$ generated by the function $M(u) = e^{u^2} -1$ coincides with the Marcinkiewicz 
space $M_{1, v}^{(*)}$ with $v(t) = \log_2^{-1/2}(2/t)$ (cf. \cite[Lemma 3.2]{As09}). Recalling that $G$ is the closure of $L_{\infty}$ 
in $M_{1, v}^{(*)}$ we note that the embedding  $G \subset M_{p, w}^{(*)}$ holds if and only if (\ref{10}) is satisfied. Therefore, 
by already mentioned Rodin-Semenov theorem (cf. \cite{RS75}; see also \cite[Theorem 2.b.4]{LT}), the Rademacher 
functions span $l_2$ in $M_{p, w}^{(*)}$ if and only if (\ref{10}) holds. 

Moreover, it is instructive to compare the behaviour of Rademacher sums in the spaces $M_{1,w}^{(*)}, M_{1,w}$ and
$K_{1, w}$ in the case when $w(t) = \log_2^{-1/q} (2/t)$, where $q > 2$. Then (\ref{10}) does not hold and 
$$
\Big\|\sum_{k=1}^{\infty} a_k r_k \Big\|_{M_{1, w}^{(*)}} \approx \| \{a_k\}_{k=1}^{\infty} \|_{l_2} + \sup_{m \in \mathbb N} \, m^{-1/q} \, 
\sum_{k=1}^m a_k^{\ast},
$$
where $\{a_k^{\ast}\}$ is the non-increasing rearrangement of $\{|a_k|\}_{k=1}^{\infty}$ (cf. Rodin-Semenov \cite[p. 221]{RS75} and Pisier \cite{Pi81}; see also Marcus-Pisier \cite[pp. 277-278]{MP84}),
$$
\Big\| \sum_{k=1}^{\infty} a_k r_k \Big\|_{M_{1, w}} \approx \| \{a_k\}_{k=1}^{\infty} \|_{l_2} + \sup_{m \in \mathbb N} \, m^{-1/q} \, 
\sum_{k=1}^m |a_k | ~~~ {\rm by ~~ (\ref{8})}, ~ {\rm and}
$$
$$
\Big\| \sum_{k=1}^{\infty} a_k r_k \Big\|_{K_{1, w}} \approx \| \{a_k\}_{k=1}^{\infty} \|_{l_2} + \sup_{m \in \mathbb N} \, m^{-1/q} \, \Big| \sum_{k=1}^ma_k \Big| ~~  {\rm (cf. ~ \cite[Theorem ~ 2]{AM10}).}
$$
}
\end{remark}

Now, we pass to studying the problem of complementability of the closed linear span ${\cal R}_p: = [r_n]_{n=1}^\infty$ in the space 
$M_{p, w}$. 
Since the results turn out to be different for $p > 1$ and $p = 1$, we consider these cases separately.

\section{  Complementability of Rademacher subspaces \newline
\hspace{9mm} in Morrey spaces $M_{p, w}$ for $p > 1$}

\begin{theorem} \label{Thm2}
Let $1<p<\infty$. The subspace ${\cal R}_p$ is complemented in the Morrey space $M_{p, w}$ if and only if condition (\ref{10}) holds.
\end{theorem}

To prove this theorem we will need the following auxiliary assertion.

\begin{proposition} \label{Pro2}
If condition (\ref{10}) does not hold, then the subspace ${\cal R}_p$ contains a complemented 
(in ${\cal R}_p$) subspace isomorphic to $c_0$.
\end{proposition}

\begin{proof}
Since $w$ is quasi-concave, by the assumption, we have
\begin{equation} \label{12}
\limsup_{n \rightarrow \infty} w(2^{-n}) \sqrt{n} = \infty.
\end{equation}
We select an increasing sequence of positive integers as follows. Let $n_1$ be the least positive integer satisfying 
the inequality $w(2^{-{n_1}}) \sqrt{n_1} \geq 2$. As it is easy to see $w(2^{-{n_1}}) \sqrt{n_1} < 2^2$. By induction, assume that 
the numbers $n_1 < n_2 < \ldots < n_{k-1}$ are chosen. Applying (\ref{12}), we take for $n_k$ the least positive 
integer such that
\begin{equation} \label{13}
w(2^{-n_k}) \sqrt{n_k - n_{k-1}} \geq 2^k.
\end{equation}
Then, obviously,
\begin{equation} \label{14}
w(2^{-n_k}) \sqrt{n_k - n_{k-1}} < 2^{k+1}.
\end{equation}
Thus, we obtain a sequence $0 = n_0 < n_1 < \ldots$ satisfying inequalities (\ref{13}) and (\ref{14}) for 
all $k \in \mathbb N$.
Let us consider the block basis $\{v_k\}_{k=1}^{\infty}$ of the Rademacher system defined as follows:
\begin{equation*}
v_k = \sum_{i = n_{k-1}+1}^{n_k} a_i\, r_i, ~~ {\rm where} ~~ a_i = \dfrac{1}{(n_k - n_{k-1}) \, w(2^{-n_k})} ~~ {\rm for} ~~ n_{k-1} < i \leq n_k.
\end{equation*}
Let us recall that, by Theorem \ref{Thm1}, if $R = \sum_{k=1}^{\infty} b_k \, r_k$, then 
$\| R \|_{M_{p, w}} \approx \| R \|_{l_2} + \| R \|_{w}$, where 
$$
\| R \|_{l_2} = \Big(\sum_{k=1}^{\infty} b_k^2\Big)^{1/2} ~~ {\rm and} ~~ 
\| R \|_{w} = \sup_{m \in \mathbb N} \Big( w(2^{-m}) \sum_{k=1}^m |b_k|\Big).
$$
Now, we estimate the norm of $v_k,$ $k=1,2,\dots,$ in $M_{p, w}$. At first, by (\ref{13}), 
\begin{equation} \label{15}
\| v_k \|_{l_2} = \Big( \sum_{i=n_{k-1}+1}^{n_k} a_i^2 \Big)^{1/2} =
\dfrac{1}{\sqrt{n_k - n_{k-1}} \, w(2^{-n_k})} \leq 2^{-k}, ~ k = 1, 2, \ldots  
\end{equation} 
Moreover, taking into account (\ref{13}), (\ref{14}) and 
the choice of $n_k$, for every $k \in \mathbb N$ and $n_{k-1} < i \leq n_k$ we have
\begin{equation*}
w(2^{-i})\, \sum_{j= n_{k-1} + 1}^i a_j = \dfrac{w(2^{-i}) (i-n_{k-1})}{(n_k - n_{k-1}) w(2^{-n_k})} \leq \dfrac{2^{k+1} 
\sqrt{i-n_{k-1}}}{2^k \sqrt{n_k - n_{k-1}}} \leq 2.
\end{equation*}
Therefore, $\| v_k \|_{w} \leq 2$ for $k \in \mathbb N$ and combining this with (\ref{15}) we obtain $\| v_k \|_{M_{p, w}} \leq C$ for 
$k \in \mathbb N$.

On the other hand, by Theorem \ref{Thm1}, 
\begin{equation} \label{15extra}
\| v_k \|_{M_{p, w}} \geq c \, w(2^{-n_k}) \, \sum_{ i=n_{k-1}+1}^{n_k} a_i = c 
\end{equation}
for some constant $c > 0$ and every $k \in \mathbb N$. Thus, $\{v_k\}_{k=1}^{\infty}$ is a semi-normalized block basis of 
$\{r_k\}_{k=1}^{\infty}$ in $M_{p, w}$.

Further, let us select a subsequence $\{m_i\} \subset \{n_k\}$ such that 
\begin{equation} \label{16}
w(2^{-m_{i+1}}) \leq \frac{1}{2} w(2^{-m_i}), ~~ i = 1, 2, \ldots
\end{equation} 
and denote by $\{u_i\}_{i=1}^{\infty}$ the corresponding subsequence of $\{v_k\}_{k=1}^{\infty}$. Then, $u_i$ can be represented 
as follows:
$$
u_i = \sum_{k=l_i}^{m_i} a_k \, r_k, ~~ {\rm where} ~~ l_i = n_{j_i - 1} + 1, ~m_i = n_{j_i}, ~ j_1 < j_2 < \ldots .
$$
Moreover, from the above $\{u_i\}$ is a semi-normalized sequence in $M_{p, w}$ and 
\begin{equation} \label {17}
\| u_i \|_{l_2} \leq 2^{-i} ~~ {\rm for} ~~  i = 1, 2, \ldots.
\end{equation}
We show that the sequence $\{u_i\}_{i=1}^{\infty}$ is equivalent in $M_{p, w}$ to the unit vector basis of $c_0$. 

Let $f = \sum_{i=1}^{\infty} \beta_i \, u_i, \beta_i \in \mathbb R$. Then, we have
$$
f = \sum_{i=1}^{\infty} \beta_i \sum_{k=l_i}^{m_i} a_k \, r_k = \sum_{k=1}^{\infty} \gamma_k \, r_k, 
$$
where $\gamma_k = \beta_i \, a_k, l_i \leq k \leq m_i, i = 1, 2, \ldots$ and $\gamma_k = 0$ if $k \not \in \cup_{i = 1}^{\infty} [l_i, m_i]$.
To estimate $\| f \|_{w}$, assume, at first, that $m_s \leq q < l_{s+1}$ for some $s \in \mathbb N$. Then, 
\begin{equation*}
\sum_{k=1}^q |\gamma_k| = \sum_{i=1}^s |\beta_i|  \sum_{k=l_i}^{m_i} a_k = \sum_{i=1}^s |\beta_i| \, \dfrac{1}{w(2^{-m_i})} \leq \| (\beta_i) \|_{c_0} \, \sum_{i=1}^s  \dfrac{1}{w(2^{-m_i})}, 
\end{equation*}
and from (\ref{16}) it follows that
\begin{equation*}
w(2^{-q}) \, \sum_{k=1}^q |\gamma_k| \leq \| (\beta_i) \|_{c_0} \, \sum_{i=1}^s  \dfrac{w(2^{-m_s})}{w(2^{-m_i})} \leq 
\| (\beta_i) \|_{c_0} \, \sum_{i=0}^{\infty} 2^{-i} = 2\, \| (\beta_i) \|_{c_0}.
\end{equation*}
Otherwise, we have $l_s \leq q < m_s, s \in \mathbb N$. Then, similarly,
\begin{eqnarray*}
\sum_{k=1}^q |\gamma_k| 
&\leq& 
\left(  \sum_{i=1}^{s-1} \dfrac{1}{w(2^{-m_i})} +  \sum_{k = l_s}^q a_k \right) \,  \| (\beta_i) \|_{c_0} \\
&=&
\left(  \sum_{i=1}^{s-1} \dfrac{1}{w(2^{-m_i})} + \dfrac{q - l_s +1}{(m_s - l_s +1) \, w(2^{-m_s})} \right) \,  \| (\beta_i) \|_{c_0}. 
\end{eqnarray*}
Since $m_s = n_{j_s}$ and $l_s = n_{j_s-1} + 1$ for some $j_s \in \mathbb N$, in view of (\ref{13}), (\ref{16}) and the choice of 
$n_{j_s}$, we obtain
\begin{eqnarray*}
w(2^{-q}) \, \sum_{k=1}^q |\gamma_k| 
&\leq& 
\left(  \sum_{i=1}^{s-1} \dfrac{w(2^{-m_{s-1}})}{w(2^{-m_i})} + \dfrac{w(2^{-q}) (q - l_s + 1)}{(m_s - l_s + 1) \, w(2^{-m_s})} \right) \,  \| (\beta_i) \|_{c_0} \\
&\leq&
\left(  \sum_{i=0}^{\infty} 2^{-i} + \dfrac{2^{j_s+1} \sqrt{q - l_s + 1}}{ 2^{j_s} \sqrt{m_s - l_s + 1}} \right) \,  \| (\beta_i) \|_{c_0} \leq 4  \,  \| (\beta_i) \|_{c_0}.
\end{eqnarray*}
Combining this with the previous estimate, we obtain that $\| f \|_w \leq 4  \,  \| (\beta_i) \|_{c_0}$. On the other hand, from (\ref{17}) it 
follows that $\| f \|_{l_2} \leq \| (\beta_i) \|_{c_0}$. Therefore, again by Theorem \ref{Thm1},
$$
\| f \|_{M_{p, w}} \leq C \, \big( \| f \|_{l_2} + \| f \|_w \big) \leq 5 \, C \, \| (\beta_i) \|_{c_0}.
$$
In opposite direction, taking into account the fact that $\{u_i\}$ is an unconditional sequence in $M_{p, w}$, by \eqref{15extra},
we obtain
$$
\| f \|_{M_{p, w}} \geq c'\,\sup_{i \in \mathbb N} |\beta_i| \, \| u_i \|_{M_{p, w}} \geq c^{\prime}c \, \| (\beta_i) \|_{c_0},
$$
for some constant $c^{\prime} > 0$. Thus, we have proved that $E: = [u_n]_{M_{p, w}} \simeq c_0$. Since ${\cal R}_p$ is separable, 
Sobczyk's theorem (see, for example, \cite[Corollary 2.5.9]{AK06}) implies that $E$ is a complemented subspace in ${\cal R}_p$.
\end{proof}

\proof[Proof of Theorem \ref{Thm2}] At first, let us assume that relation (\ref{10}) holds. Then, by Corollary \ref{Cor2}, 
${\cal R}_p \simeq l_2$. Therefore, since $M_{p, w} \stackrel {1}\hookrightarrow L_p$, by the Khintchine inequality, the orthogonal 
projection $P$ generated by the Rademacher system satisfies the following:
$$
\| Pf \|_{M_{p, w}} \approx \| Pf \|_{L_p} \leq \| P\|_{L_p \rightarrow L_p} \| f \|_{L_p} \leq \| P\|_{L_p \rightarrow L_p} \| f \|_{M_{p, w}}, 
$$
because $P$ is bounded in $L_p, 1 < p < \infty$. Hence, $P: M_{p, w} \rightarrow M_{p, w}$ is bounded.

Conversely, we argue in a similar way as in the proof of Theorem 4 in \cite{ALM11}. Suppose that the subspace 
${\cal R}_p = [r_n]_{n=1}^\infty$ is complemented in $M_{p, w}$ and let $P_1: M_{p, w} \rightarrow M_{p, w}$ be a 
bounded linear projection whose range is ${\cal R}_p$. By Proposition \ref{Pro2}, there is a subspace $E$ complemented 
in ${\cal R}_p$ and such that $E \simeq c_0$. Let $P_2: {\cal R}_p \rightarrow E$ be a bounded linear projection. Then 
$P: = P_2 \circ P_1$ is a linear projection bounded in $M_{p, w}$ whose image coincides with $E$. Thus, $M_{p, w}$
contains a complemented subspace $E \simeq c_0$. 

Since $M_{p, w}$ is a conjugate space (more precisely, $M_{p, w} = (H^{q, u})^{\ast}$, where $H^{q, u}$ is the ``block space" 
and $1/p + 1/q = 1$ -- see, for example, \cite[Proposition 5]{Zo86}; see also \cite{BRV99} and \cite{Na11}), this contradicts 
the well-known result due to Bessaga-Pe{\l}czy\'nski saying that arbitrary conjugate space cannot contain a complemented 
subspace isomorphic to $c_0$ (see \cite[Corollary 4]{BP58} and \cite[Theorem~4 and its proof]{BP58S}). This contradiction proves the theorem.
\endproof

\section{Rademacher subspace ${\cal R}_1$ is not complemented in Morrey space $M_{1, w}$}

\begin{theorem} \label{Thm3}
For every quasi-concave weight $w$ the subspace ${\cal R}_1$ is not complemented in the Morrey space $M_{1, w}$.
\end{theorem}

In the proof we consider two cases separately, depending if the condition (\ref{10}) is satisfied or not. 

\begin{proof}[Proof of Theorem \ref{Thm3}: the case when (\ref{10}) does not hold] 
On the contrary, we suppose that ${\cal R}_1$ is complemented in $M_{1, w}$. Then, if $Q$ is a bounded 
linear projection from $M_{1, w}$ onto ${\cal R}_1$, by Theorem \ref{Thm1}, for every $p \in (1, \infty)$ and $f \in M_{p, w}$, we have 
\begin{equation*}
\| Q f \|_{M_{p, w}} \approx \| Q f \|_{M_{1, w}} \leq \| Q\| \, \| f \|_{M_{1, w}} \leq \| Q \| \, \| f \|_{M_{p, w}}.
\end{equation*}
Thus, $Q$ is a bounded projection from $M_{p, w}$ onto ${\cal R}_p$, which contradicts Theorem \ref{Thm2}.
\end{proof}

To prove the assertion in the case when (\ref{10}) holds, we will need auxiliary results.
\noindent
Let $M_{p, w}^d$ be the dyadic version of the space $M_{p, w}$, $1 \leq p < \infty$, consisting of all measurable functions 
$f: [0, 1] \rightarrow {\mathbb R}$ such that
\begin{equation*}
\|f\|_{M_{p, w}^d} = \sup_{I \in {\cal D}} ~ w(|I|) \left( \frac{1}{|I|}\int_I | f(t) |^p\,dt \right)^{1/p} < \infty.
\end{equation*}

\begin{lemma} \label{Lem1}
For every $1 \leq p < \infty$ $M_{p, w} = M_{p, w}^d$ and
\begin{equation} \label{18}
\|f\|_{M_{p, w}^d} \leq \|f\|_{M_{p, w}} \leq 4 \, \|f\|_{M_{p, w}^d}.
\end{equation}
\end{lemma}
\begin{proof}
The left-hand side inequality in (\ref{18}) is obvious. To prove the right-hand side one, we observe that for any interval 
$I \subset [0, 1]$ we can find adjacent dyadic intervals $I_1$ and $I_2$ of the same length such that 
$I \subset I_1 \cup I_2$ and $\frac{1}{2} |I_1| \leq | I | \leq 2 |I_1|$. Then, by the quasi-concavity of $w$,
\begin{eqnarray*}
&& w(| I |) \left( \frac{1}{| I |} \int_{I} |f(t)|^p \, dt \right)^{1/p} =
\frac{w(| I |)}{| I |} \Big( | I |^{p-1} \int_{I} |f(t)|^p \, dt \Big)^{1/p} \\
&\leq&
\frac{w(\frac{1}{2} |I_1|)}{\frac{1}{2}|I_1|} \Big[ 2^{p-1} |I_1|^{p-1} \Big( \int_{I_1} |f(t)|^p \, dt + \int_{I_2} |f(t)|^p \, dt \Big) \Big]^{1/p}\\
&\leq&
2^{2-1/p} \, w(|I_1|) \left( \frac{1}{|I_1|} \int_{I_1} |f(t)|^p \, dt + \frac{1}{|I_2|} \int_{I_2} |f(t)|^p \, dt \right)^{1/p} \\
&\leq&
4 \, \sup_{I \in {\cal D}} w(|I|) \left( \frac{1}{|I|}  \int_{I} |f(t)|^p \, dt \right)^{1/p} = 4 \, \|f\|_{M_{p, w}^d}.
\end{eqnarray*}
Taking the supremum over all intervals $I \subset [0, 1]$, we obtain the right-hand side inequality in (\ref{18}).
\end{proof}

Let $P$ be the orthogonal projection generated by the Rademacher sequence, i.e., 
\begin{equation*} \label{19}
Pf(t): = \sum_{k=1}^{\infty} \int_0^1 f(s) r_k(s) \, ds \cdot r_k(t).
\end{equation*}

\begin{proposition} \label{Pro3}
Let $1 \leq p < \infty$. If ${\cal R}_p$ is a complemented subspace in $M_{p, w}$, then the projection $P$ is bounded in 
$M_{p, w}$.
\end{proposition}
\begin{proof}
By Lemma \ref{Lem1}, it is sufficient to prove the same assertion for the dyadic space $M_{p, w}^d$. We almost repeat 
the arguments from the proof of the similar result for r.i. function spaces (see \cite{RS79} or \cite[Theorem 3.4]{As09}).

Let $t=\sum\limits_{i=1}^{\infty} \alpha_{i} 2^{-i}$ and $u=\sum\limits_{i=1}^{\infty} \beta_{i} 2^{-i}$ $(\alpha_{i},
\beta_{i}=0,1)$ be the binary expansion of the numbers $t,u \in [0,1]$.
Define the following operation:
$$
t\oplus u=\sum_{i=1}^{\infty} 2^{-i} [(\alpha_{i} + \beta_{i} ) ~ {\rm mod} ~ 2]. 
$$
One can easily verify that this operation transforms the segment $[0, 1]$ into a compact Abelian group.
For every $u\in [0,1]$, the transformation $w_{u}(s)=s \oplus u$ preserves the Lebesgue measure on $[0, 1]$, 
i.e., for any measurable $E\subset [0,1]$, its inverse image $w_{u}^{-1}(E)$ is measurable and $m(w_{u}^{-1}(E))=m(E)$. 
Moreover, $w_{u}$ maps any dyadic interval onto some dyadic interval. Hence, the operators 
$T_{u}f = f\circ w_{u}$ $(0\le u \le 1)$ act isometrically in $M_{p, w}^d$. From the definition of the Rademacher functions 
it follows that the subspace ${\cal R}_p$ is invariant with respect to these operators. Therefore,
by the Rudin theorem (see \cite[Theorem 5.18, pp. 134-135]{Ru91}), there exists a bounded linear projector $Q$ acting from $M_{p, w}^d$
onto ${\cal R}_p$ and commuting with all operators $T_{u} \;\; (0\le u \le 1)$. We show that $Q = P$.

First of all, the projector $Q$ has the representation
\begin{equation} \label{20}
Q f(t)= \sum\limits_{i=1}^{\infty} Q_{i}(f) \, r_{i}(t),
\end{equation}
where by Theorem \ref{Thm1}, $Q_{i} ~ (i = 1, 2, \ldots)$ are linear bounded functionals on $M_{p, w}^d$. It is obvious that
\begin{equation} \label{21}
Q_{i}(r_{j})=
\left\{
 \begin{array}{rcl}
 1 ~{\rm if} ~ i=j,\\
 0 ~ {\rm if} ~ i\not=j.\\
 \end{array}
   \right.
 \end{equation}
Consider the sets
$$
U_i =\Big\{ u\in [0,1]: u=\sum_{j=1}^{\infty} \alpha_{j}2^{-j}, \;\; \alpha_{i}=0\Big\}, ~~ U_i^c = [0,1]\backslash U_i.
$$ 
One can check that
$$
r_{i}(t\oplus u)=
\left\{
 \begin{array}{rcl}
 r_{i}(t) ~{\rm if} ~ u\in U_i,\\
 -r_{i}(t) ~~ {\rm if} ~ u \in U_i^c.  \\
\end{array}
\right.
$$
Due to the relation $T_{u}Q=QT_{u}$ $(0\le u \le 1)$ this implies
$$
Q_{i}(T_uf)=
\left\{
 \begin{array}{rcl}
 Q_{i}(f) ~{\rm if} ~ u\in U_i,\\
 -Q_{i}(f) ~~{\rm if} ~ u\in U_i^c.  \\
 \end{array}
 \right.
 $$
Taking into account that $m(U_i)= m(U_i^c)= 1/2$, we find that
$$
 \int_{U_i} Q_{i}(T_uf)\, du = \frac{1}{2} Q_{i}(f) ~~{\rm and} ~~  \int_{U_i^c} Q_{i}(T_uf)\, du = - \frac{1}{2} Q_{i}(f).
$$
Thanks to the boundedness of $Q_{i}$ , this functional can be moved outside the integral; therefore, we obtain
\begin{equation}\label{22} 
Q_{i}(f) = Q_i \Bigl ( \int_{U_i} T_uf\,du - \int _{U_i^c}T_uf\,du \Bigl).
\end{equation}
Since
$$ 
\{ s\in [0,1]:\, s = t\oplus u, \;\; u\in U_i \}= 
\left\{
 \begin{array}{rcl}
 U_i ~ {\rm if} ~ t\in U_i,\\
 U_i^c ~ {\rm if} ~ t \in U_i^c,  \\
\end{array}
\right.
$$
$$ 
\{ s\in [0,1]:\, s=t\oplus u, \;\; u\in U_i^c \}=
\left\{
 \begin{array}{rcl}
U_i^c ~ {\rm if} ~ t\in U_i,\\
U_i  ~ {\rm if} ~ t \in U_i^c,  \\
\end{array}
\right.
$$
and the transformation $\omega_u$ preserves the Lebesgue measure on $[0,1]$, we have 
$$
\int_{U_i} T_uf(t)du\;=\;\int _{U_i} f(s)\,ds \cdot \chi_{U_i}(t)\;+\; \int_{U_i^c} f(s)\,ds \cdot\chi_{U_i^c}(t)
$$
and
$$
\int _{U_i^c} T_uf(t)\,du\;=\;\int_{U_i^c}f(s)\,ds \cdot \chi_{U_i}(t)\;+\;\int_{U_i}f(s)\,ds \cdot\chi_{U_i^c}(t). 
$$
It is easy to see that $r_{i}(t)= \chi_{U_i}(t) - \chi_{U_i^c}(t)$. Therefore, from the last two relations it follows that
$$
\int\limits _{U_i} T_uf(t)\,du - \int_{U_i^c}T_uf(t)\,du = \int_0^1 f(s) r_{i}(s)\,ds \cdot r_{i}(t).
$$ 
This and (\ref{20})--(\ref{22}) yield
$$
Q_{i}(f) = \int_0^1 f(s) r_{i}(s)\,ds, ~~ i = 1, 2, \ldots,
$$
i.e., $Q = P$, and Proposition \ref{Pro3} is proved.
\end{proof}

The following result, in fact, is known. However, we provide its proof for completeness.

\begin{lemma} \label{Lem2}
Suppose that the Rademacher sequence is equivalent in a Banach function lattice $X$ on $[0, 1]$ to the unit vector basis in $l_2$, 
i.e., for some constant $C > 0$ and all $a = (a_k)_{k=1}^{\infty} \in l_2$ 
\begin{equation} \label{23}
C^{-1} \, \| a \|_{l_2} \leq \Big\| \sum_{k=1}^{\infty} a_k \, r_k \Big\|_{X} \leq C \, \| a \|_{l_2}.
\end{equation}
Moreover, let $\{r_k\} \subset X^{\prime}$, where $X^{\prime}$ is the K\"othe dual space for $X$. Then, the orthogonal 
projection $P$ is bounded in $X$ if and only if there exists a constant $C_1 > 0$ such that for every $a = (a_k)_{k=1}^{\infty} \in l_2$ 
\begin{equation} \label{24}
\Big\| \sum_{k=1}^{\infty} a_k \, r_k \Big\|_{X^{\prime}} \leq C_1 \, \| a \|_{l_2}.
\end{equation}
\end{lemma}
\begin{proof}
First, suppose that (\ref{24}) holds. For arbitrary $f \in X$, we set
$$
c_k(f) = \int_0^1 f(s) r_k(s) \, ds, ~ k = 1, 2, \ldots
$$
By (\ref{24}), for every $n \in \mathbb N$, we have
\begin{eqnarray*}
\sum_{k=1}^n c_k(f)^2 
&=&
\int_0^1 f(s) \sum_{k=1}^n c_k(f) r_k(s) \, ds \leq 
\| f\|_X \Big\| \sum_{k=1}^n c_k(f) r_k \Big\|_{X^{\prime}} \\
&\leq&
C_1 \, \| f\|_X \Big( \sum_{k=1}^n c_k(f)^2 \Big)^{1/2},
\end{eqnarray*}
and therefore, taking into account (\ref{23}), we obtain
\begin{equation*}
\| Pf \|_X \leq C \Big( \sum_{k=1}^{\infty} c_k(f)^2 \Big)^{1/2} \leq C \cdot C_1 \| f \|_X.
\end{equation*}
Thus, $P$ is bounded in $X$.

Conversely, if $P$ is a bounded projection in $X$, then from (\ref{23}) it follows that
\begin{eqnarray*}
\int_0^1 f(t) \sum_{k=1}^n a_k r_k(t) \, dt
&=&
\sum_{k=1}^n a_k \cdot c_k(f) \leq \| a \|_{l_2} \, \Big( \sum_{k=1}^n c_k(f)^2 \Big)^{1/2} \\
&\leq&
C \, \| a \|_{l_2} \, \| Pf \|_X \leq C \, \| P \|_{X \rightarrow X}  \, \| a \|_{l_2} \, \| f \|_X 
\end{eqnarray*}
for each $n \in \mathbb N$, all $a = (a_k)_{k=1}^{\infty} \in l_2$ and $f \in X$. Hence,
\begin{equation*}
\Big\| \sum_{k=1}^{\infty} a_k \, r_k \Big\|_{X^{\prime}} \leq C  \, \| P \|_{X \rightarrow X} \, \| a \|_{l_2},
\end{equation*}
and (\ref{24}) is proved.
\end{proof}

\proof[Proof of Theorem \ref{Thm3}: the case when (\ref{10}) holds] 
In view of Lemmas \ref{Lem1}, \ref{Lem2} and Proposition \ref{Pro2} it is sufficient to prove that
\begin{equation} \label{25}
\limsup_{n \rightarrow \infty} \frac{1}{ \sqrt{n}} \, \Big\| \sum_{k=1}^n r_k \Big\|_{(M_{1, w}^d)^{\prime}} = \infty.
\end{equation}
For every $m \in \mathbb N$ such that $\sqrt{m/2} \in \mathbb N$ we consider the set
$$
E_m: = \{t \in [0, 1]: 0 \leq \sum_{k=1}^{2m} r_k(t) \leq \sqrt{m/2}\, \}.
$$
Clearly, $E_m = \cup_{k \in S_m} I_k^{2m}$, where $S_m \subset \{1, 2, \ldots, 2^{2m}\}$. Also, it is 
easy to see that $|E_m| \rightarrow 0$ as $m \rightarrow \infty$. Denoting
$$
f_m: = \frac{1}{w(|E_m|)} \, \chi_{E_m}, ~ m \in \mathbb N,
$$
we show that
\begin{equation} \label{26}
\| f_m \|_{M_{1, w}^d} \leq 1 ~~ {\rm for ~ all} ~~m \in \mathbb N.
\end{equation}
In fact, let $I$ be a dyadic interval from $[0, 1]$. Clearly, we can assume that  $I \cap E_m \not = \emptyset$. 
Then, by using the quasi-concavity of $w$, we have
\begin{equation*}
\frac{w(|I|)}{|I|} \int_I |f_m(t)|\, dt = \frac{w(|I|)}{|I|} \cdot \frac{|I \cap E_m|)}{w(|E_m|)} 
\leq \frac{w(|I|)}{|I|} \cdot \frac{|I \cap E_m|}{w(|I \cap E_m|)} \leq 1,
\end{equation*}
and (\ref{26}) is proved.

From (\ref{26}) it follows that
\begin{eqnarray*}
 \Big\| \sum_{k=1}^{2m} r_k \Big\|_{(M_{1, w}^d)^{\prime}} 
&\geq&
\int_0^1 \Big| \sum_{k=1}^{2m} r_k(t) \Big| \cdot f_m(t) \, dt = 
\frac{1}{w(|E_m|)} \int_{E_m}  \Big| \sum_{k=1}^{2m} r_k(t) \Big| \, dt \\
&=& 
\frac{1}{w(|E_m|)} \, 2^{-2m} \sum_{i \in S_m}  \Big| \sum_{k=1}^{2m} \varepsilon_k^i \Big|,
\end{eqnarray*}
where $ \varepsilon_k^i  = {\rm sign} ~ r_k |_{\Delta_i^{2m}}$, $k = 1, 2, \ldots, 2m, ~i \in S_m.$
Denoting $\sigma_m:= \sum_{i \in S_m} | \sum_{k=1}^{2m} \varepsilon_k^i |,$ by the definition of $E_m$, we obtain
\begin{equation} \label{27}
\sigma_m = 2 \cdot \sum_{m - \sqrt{m/2} \leq k \leq m} C_k^{2m} (m - k) = 2 \cdot \sum_{k=1}^{ \sqrt{m/2}} C_{m - k}^{2m} \cdot  k, 
\end{equation}
where $C^n_i=\frac{n!}{i!(n-i)!},$ $n=1,2,\dots,$ $i=0,1,\dots,n.$
Let us estimate the ratio ${C_{m - k}^{2m} }/{C_m^{2m} }$ for $1 \leq k \leq \sqrt{m/2}$ from below. At first,
\begin{eqnarray*}
\dfrac{C_{m - k}^{2m} }{C_m^{2m} }
&=&
\frac{(m!)^2}{(m-k)! (m+k)!} = \frac{(m-k+1) \cdot \ldots \cdot (m-1) \cdot m}{(m+1) \cdot \cdot \ldots \cdot (m+k-1) \cdot (m+k)} \\
&=&
\frac{m}{m+k}  \cdot\frac{(m-k+1) \cdot \ldots \cdot (m-1)}{(m+1) \cdot \ldots \cdot (m+k-1)} = \frac{m}{m+k} \cdot 
\prod_{j=1}^{k-1} \frac{1- \frac{j}{m}}{1 + \frac{j}{m}}\\
&=&
\frac{m}{m+k} \cdot \exp \Big( \sum_{j=1}^{k-1} \log \frac{1- \frac{j}{m}}{1 + \frac{j}{m}}\Big).
\end{eqnarray*}
Next, we will need the following elementary inequality
\begin{equation} \label{28}
\log \frac{1- t}{1 + t} + 2 \, t + 2 \, t^3 \geq 0 ~~ {\rm for ~ all} ~~ 0 \leq t \leq \frac{1}{2}.
\end{equation}
Indeed, we set
$$
\varphi(t):= \log \frac{1- t}{1 + t} + 2 \, t + 2 \, t^3.
$$
Then, $\varphi(0) = 0$. Moreover, for all $t \in [0, 1/2]$ we have
$$
\varphi^{\prime}(t) = - \frac{2}{1-t^2} + 2 + 6 t^2 = \frac{2 t^2 (2 - 3 t^2)}{1-t^2} \geq 0.
$$
Thus, $\varphi(t)$ increases on the interval $[0, 1/2]$, and (\ref{28}) is proved. 

From the above formula, inequality (\ref{28}) and the condition $1 \leq k \leq \sqrt{m/2}$ we obtain
\begin{eqnarray*}
\dfrac{C_{m - k}^{2m} }{C_m^{2m} }
&\geq&
\frac{m}{m+k} \exp \Big( - \frac{2}{m} \sum_{j=1}^{k-1} j - \frac{2}{m^3} \sum_{j=1}^{k-1} j^3 \Big) \\
&=& 
\frac{m}{m+k} \exp \Big( \frac{-k(k-1)}{m} \Big) \exp \Big( \frac{- (k-1)^2 k^2}{2 m^3} \Big) \\
&\geq&
\frac{1}{2} \, \exp \Big(- \frac{k^2}{m} - \frac{1}{m} \Big).
\end{eqnarray*}
Combining this estimate with equality (\ref{27}), we infer
\begin{equation} \label{29}
\sigma_m = 2 \cdot \sum_{k=1}^{ \sqrt{m/2}} \frac{C_{m - k}^{2m}}{C_m^{2m}} \cdot  k \cdot C_m^{2m} 
\geq C_m^{2m} \cdot e^{-1/m} \cdot \sum_{k=1}^{ \sqrt{m/2}} e^{- \frac{k^2}{m}} \cdot k.
\end{equation}
The function $\psi(u) = e^{-\frac{u^2}{m}} \cdot u$ increases on the interval $[0,  \sqrt{m/2}]$ because of
$$
\psi^{\prime}(u) = e^{-\frac{u^2}{m}} + u e^{-\frac{u^2}{m}} (- {2u}/{m}) = e^{-\frac{u^2}{m}} (1 - {2 u^2}/{m}) \geq 0
$$
for $0 \leq u \leq \sqrt{m/2}$. Therefore, 
\begin{eqnarray*}
\sum_{k=1}^{ \sqrt{m/2}} e^{- \frac{k^2}{m}} \cdot k
&>&
\sum_{k=1}^{ \sqrt{m/2}} \int_{k-1}^k e^{- \frac{u^2}{m}} \cdot u \, du = \frac{m}{2} (1 - \frac{1}{\sqrt{e}}) \geq \frac{1}{3} m.
\end{eqnarray*}
Moreover, an easy calculation, by using the Stirling formula, shows that
$$
\lim_{m \rightarrow \infty} C_m^{2m} 4^{-m} \sqrt{\pi m} = 1.
$$
Thus, from the above and (\ref{29}) it follows that
\begin{eqnarray*}
 \Big\| \sum_{k=1}^{2m} r_k \Big\|_{(M_{1, w}^d)^{\prime}} 
&\geq&
\frac{1}{w(|E_m|)} \, 2^{-2m} \sum_{i \in S_m}  \Big| \sum_{k=1}^{2m} \varepsilon_k^i \Big| = 
\frac{1}{w(|E_m|)} \, 2^{-2m} \, \sigma_m \\
&\geq&
\frac{1}{w(|E_m|)} \, 2^{-2m} \cdot  C_m^{2m} 
\cdot e^{-1/m} \cdot \sum_{k=1}^{ \sqrt{m/2}} e^{- \frac{k^2}{m}} \cdot k\\
&\geq&
\frac{1}{w(|E_m|)} \, 4^{-m} \cdot  C_m^{2m} \cdot e^{-1/m} \cdot \frac{1}{3} m \approx \frac{\sqrt{m}}{3 \sqrt{\pi} \, w(|E_m|)} 
\end{eqnarray*}
for all $m \in \mathbb N$ such that $\sqrt{m/2} \in \mathbb N$. Since $|E_m| \rightarrow 0$, then by (\ref{10}) $w(|E_m|) \rightarrow 0$ 
as $m \rightarrow \infty$. Hence, the preceding inequality implies (\ref{25}) and the proof is complete.
\endproof

\section{Structure of Rademacher subspaces in Morrey spaces}

Applying Theorem \ref{Thm1} allows us also to study the geometric structure of Rademacher subspaces in Morrey 
spaces $M_{p, w}$.

\begin{theorem} \label{Thm4}
Let $1 \leq p < \infty$ and $\lim_{t \rightarrow 0^+} w(t) = 0$. Then every infinite-dimensional subspace of 
${\cal R}_p$ is either isomorphic to $l_2$ or contains a subspace, which is isomorphic to $c_0$ and is complemented in ${\cal R}_p$.
\end{theorem}

The following two propositions are main tools in the proof of the above theorem.

\begin{proposition} \label{Pro4}
Suppose that $1 \leq p < \infty$ and $\lim_{t \rightarrow 0^+} w(t) = 0$. Then the Rademacher functions form a shrinking basis 
in ${\cal R}_p$.
\end{proposition}
\begin{proof}
To prove the shrinking property of $\{r_n\}_{n=1}^{\infty}$ we need to show that for every $\varphi \in (M_{p, w})^*$ we have
\begin{equation} \label{30}
\| \varphi_{\big | [r_n]_{n=m}^{\infty}} \|_{ (M_{p, w})^*} \rightarrow 0 ~ {\rm as} ~ m \rightarrow \infty.
\end{equation}
Assume that (\ref{30}) does not hold. Then there exist $\varepsilon \in (0, 1), \varphi \in (M_{p, w})^*$ with 
$\| \varphi \|_{(M_{p, w})^*} = 1$, and a sequence of functions 
$$
f_n = \sum_{k=m_n}^{\infty} a_k^{m_n} r_k, ~~ {\rm where} ~~ m_1 < m_2 < \ldots, 
$$
such that $\| f_n \|_{M_{p, w}} = 1, ~ n = 1, 2, \ldots$ and
\begin{equation} \label{31}
\varphi(f_n) \geq \varepsilon ~~ {\rm for ~ all} ~~ n = 1, 2, \ldots .
\end{equation}
Let us  construct two sequences of positive integers $\{q_i\}_{i=1}^{\infty}$ and $\{p_i\}_{i=1}^{\infty}, 
1 \leq q_1 < p_1 < q_2 < p_2 < \ldots$ as follows. Setting $q_1 = m_1$, we can find $p_1 > q_1$, so that 
$\| \sum_{n=p_1+1}^{\infty} a_k^{q_1} r_k \|_{M_{p, w}} \leq {\varepsilon}/{2}$. Now, if the numbers 
$1 \leq q_1 < p_1 < q_2 < p_2 < \ldots q_{i-1} < p_{i-1}, i \geq 2,$ are chosen, we take for $q_i$ the smallest 
of numbers $m_n$, which is larger than $p_{i-1}$ such that
\begin{equation} \label{32}
w(2^{-q_{i}}) \leq \frac{1}{2} \, w(2^{-q_{i -1}}).
\end{equation}
Moreover, let $p_i > q_i$ be such that
\begin{equation} \label{33}
\Big\| \sum_{n=p_i+1}^{\infty} a_k^{q_i} r_k \Big\|_{M_{p, w}} \leq {\varepsilon}/{2}.
\end{equation}
We set $\alpha_k^i:= a_k^{q_i}$ if $q_i \leq k \leq p_i$, and $\alpha_k^i:= 0$ if $p_i < k < q_{i+1}, i = 1, 2, \ldots$. Then, the sequence
$$
u_i:= \sum_{k = q_i}^{q_{i+1} - 1} \alpha_k^i \, r_k, ~~ i = 1, 2, \ldots
$$
is a block basis of the Rademacher sequence. Moreover, by the definition of $u_i$,
\begin{equation} \label{34}
\sup_{i = 1, 2, \ldots} \| u_i \|_{M_{p, w}} \leq 2,
\end{equation}
and from the choice of the functional $\varphi$ and (\ref{33}) it follows that
\begin{equation} \label{35}
\varphi(u_i) = \varphi\Big( \sum_{k = q_i}^{p_i} a_k^{q_i} \, r_k\Big) = \varphi(f_i)  - \varphi \Big( \sum_{k = p_i + 1}^{\infty} a_k^{q_i} \, r_k\Big)
\geq  \varphi(f_i) - \Big\| \sum_{k = p_i + 1}^{\infty} a_k^{q_i} \, r_k \Big\|_{M_{p, w}} \geq \frac{\varepsilon}{2}.
\end{equation}
Let $\{\gamma_n\}_{n=1}^{\infty}$ be an arbitrary sequence of positive numbers such that
\begin{equation} \label{36}
\sum_{n=1}^{\infty} \gamma_n^2 < \infty ~~  {\rm and} ~~ \sum_{n=1}^{\infty} \gamma_n = \infty.
\end{equation}
We show that the series $\sum_{n=1}^{\infty} \gamma_n\, u_n$ converges in $M_{p, w}$. To this end, we set
$b_k:= \alpha_k^i \cdot \gamma_i$ if $q_i \leq k < q_{i+1}$. For every $m \in \mathbb N$, by Theorem \ref{Thm1}, 
\begin{equation} \label{37}
\Big\| \sum_{n=m}^{\infty} \gamma_n \, u_n \Big\|_{M_{p, w}} = \Big\| \sum_{k = q_m}^{\infty} b_k r_k \Big\|_{M_{p, w}} \approx 
\Big(\sum_{k=q_m}^{\infty} b_k^2 \Big)^{1/2} + \sup_{l \geq q_m} w(2^{-l}) \cdot \sum_{k=q_m}^l |b_k|.
\end{equation}
Let us estimate both summands from the right-hand side of (\ref{37}). At first, from (\ref{34}) and Theorem \ref{Thm1} 
it follows that
\begin{equation} \label{38}
\sum_{k=q_m}^{\infty} b_k^2 = \sum_{i=m}^{\infty} \gamma_i^2 \sum_{k=q_i}^{q_{i+1} - 1} (\alpha_k^i)^2 \leq 
C_1\, \sum_{i=m}^{\infty} \gamma_i^2.
\end{equation}
Similarly, if $q_m < \ldots < q_{m+r} \leq l < q_{m+r+1}$ for some $r = 1, 2, \ldots$, then
\begin{eqnarray*}
\sum_{k=q_m}^{l} |b_k| 
&=& 
\sum_{i=m}^{m+r-1} |\gamma_i| \sum_{k = q_i}^{q_{i+1} - 1} |\alpha_k^i| + |\gamma_{m+r}| \sum_{k = q_{m+r}}^{l}  |\alpha_k^{m+r}| \\
&\leq&
C_2 \, \Big( \sum_{i=m}^{m+r-1} \frac{|\gamma_i|}{w(2^{-q_{i+1}})} + \frac{|\gamma_{m+r}|}{w(2^{-l})}  \Big).
\end{eqnarray*}
Combining this inequality together with (\ref{32}), we obtain
\begin{eqnarray*}
w(2^{-l}) \, \sum_{k=q_m}^{l} |b_k| 
&\leq&
C_2  \, \Big( \sum_{i=m}^{m+r-1} |\gamma_i| \frac{w(2^{-q_{m+r}})}{w(2^{-q_{i+1}})} + |\gamma_{m+r}|  \Big) \\
&\leq& 
C_2  \, \Big( \sum_{i=m}^{m+r-1} |\gamma_i| \, 2^{-m-r+i+1} + |\gamma_{m+r}|  \Big) \\
&\leq& 
C_2  \, \max_{i \geq m}  |\gamma_i| \Big( \sum_{j=0}^{r-1}\, 2^{1+j - r} + 1\Big) < 3 \, C_2 \, \max_{i \geq m}  |\gamma_i|.
\end{eqnarray*}
Clearly, the latter estimate holds also in the simpler case when $q_m \leq l < q_{m+1}$. Thus, for every $m \in \mathbb N$,
\begin{equation} \label{39}
\sup_{l \geq q_m} w(2^{-l}) \, \sum_{k=q_m}^{l} |b_k|  \leq 3 \, C_2 \, \max_{i \geq m}  |\gamma_i|.
\end{equation}
From (\ref{36}) --- (\ref{39}) it follows that
the series $\sum_{n=1}^{\infty} \gamma_n \, u_n$ converges in $M_{p, w}$. At the same time, since $\varphi \in (M_{p, w})^*$, by 
(\ref{35}) and (\ref{36}), we have
$$
\varphi \Big(\sum_{n=1}^{\infty} \gamma_n \, u_n\Big) = \sum_{n=1}^{\infty} \gamma_n \, \varphi(u_n) \geq \frac{\varepsilon}{2} \sum_{n=1}^{\infty} \gamma_n = \infty,
$$
and so (\ref{30}) is proved.
\end{proof}

\begin{corollary} \label{Cor3}
Under assumptions of Proposition \ref{Pro4}: 
\begin{itemize} 
\item[(i)] $r_k \rightarrow 0$ weakly in $M_{p, w}$.
\item[(ii)] The Rademacher functions form a basis in the dual space $({\cal R}_p)^*$.
\end{itemize}
\end{corollary}
\begin{proof}
Since $\{r_n\}_{n=1}^{\infty}$ is the biorthogonal system to $\{r_n\}$ itself, (ii) follows from Proposition \ref{Pro4} 
and Proposition 1.b.1 in \cite{LT77}.
\end{proof}

\begin{proposition} \label{Pro5}
Let $1 \leq p < \infty$ and $\lim_{t \rightarrow 0^+} w(t) = 0$. Suppose that
$$
u_n = \sum_{k=m_n}^{m_{n+1} -1} a_k \, r_k, ~~ 1 = m_1 < m_2 < \ldots
$$
is a block basis such that $\| u_n \|_{M_{p, w}} = 1$ for all $ n \in \mathbb N$ and $\sum_{k=m_n}^{m_{n+1} -1} a_k^2 \rightarrow 0$ 
as $n \rightarrow \infty$. Moreover, let
\begin{equation} \label{40}
w(2^{- m_{n+1}}) \leq \frac{1}{2} \, w(2^{-m_n}), ~n = 1, 2, \ldots .
\end{equation}
Then the sequence $\{u_n\}_{n=1}^{\infty}$ contains a subsequence equivalent in $M_{p, w}$ to the unit vector basis of $c_0$.
\end{proposition}
\begin{proof}
Passing to a subsequence if it is needed, without loss of generality we may assume that
\begin{equation} \label{41}
\sum_{k=m_n}^{m_{n+1} -1} a_k^2 \leq \cdot 2^{-n}, ~~n = 1, 2, \ldots. 
\end{equation}
Suppose that $f = \sum_{n=1}^{\infty} \beta_n \, u_n \in {\cal R}_p$. Setting
$b_k = a_k \beta_i$ if $m_i \leq k < m_{i+1}, i = 1, 2, \ldots$, by Theorem \ref{Thm1}, we obtain
\begin{equation} \label{42}
\| f \|_{M_{p, w}} = \Big\| \sum_{k=1}^{\infty} b_k \, r_k \Big\|_{M_{p, w}} \approx \Big( \sum_{k=1}^{\infty} b_k^2 \Big)^{1/2} + 
\sup_{l \in \mathbb N} w(2^{-l}) \sum_{k=1}^l |b_k|.
\end{equation}
At first, by (\ref{41}),
\begin{equation*}
\sum_{k=1}^{\infty} b_k^2 = \sum_{i=1}^{\infty} \beta_i^2  \sum_{k=m_i}^{m_{i+1} -1} a_k^2 
\leq ( \sup_{i = 1, 2, \ldots} |\beta_i|)^2 \cdot \sum_{i=1}^{\infty} 2^{-i} 
\leq  \| (\beta_i)\|_{c_0}^2.
\end{equation*}
Moreover, precisely in the same way as in the proof of Proposition \ref{Pro4} from (\ref{40}) and the equalities 
$\| u_n \|_{M_{p, w}} = 1, n = 1, 2, \ldots$ it follows that for some constant $C^{\prime} > 0$
$$
 \sup_{l = 1, 2, \ldots} w(2^{-l}) \sum_{k=1}^l |b_k| \leq C^{\prime} \, \| (\beta_i)\|_{c_0}.
$$
Combining the last two inequalities together with (\ref{42}), we conclude that
$\| f \|_{M_{p, w}} \leq C \,  \, \| (\beta_i)\|_{c_0}$ for some constant $C > 0$.

Conversely, since $\{u_n\}$ is an unconditional sequence in $M_{p, w}$ and $\| u_n \|_{M_{p, w}} = 1,$ $n = 1, 2, \ldots$, by Theorem \ref{Thm1},
$\| f \|_{M_{p, w}} \geq c|\beta_i|.$ $i = 1, 2, \ldots$, with some constant $c > 0$.
Hence, $\| f \|_{M_{p, w}} \geq c\, \| (\beta_i)\|_{c_0}$, and the proof is complete.
\end{proof}

\proof[Proof of Theorem \ref{Thm4}] 
Assume that $X$ is an infinite-dimensional subspace of ${\cal R}_p$ such that for every $f=\sum_{k=1}^{\infty} b_k r_k \in X$ we have
$$
\|f\|_{M_{p, w}} \approx \Big(\sum_{k=1}^{\infty} b_k^2\Big)^{1/2},
$$
with a constant independent of $b_k,$ $k=1,2,\dots$
Then, $X$ is isomorphic to some subspace of $l_2$ and so to $l_2$ itself.

Therefore, if $X$ is not isomorphic to $l_2$, then there is a sequence $\{f_n\}_{n=1}^{\infty} \subset X, 
f_n = \sum_{k=1}^{\infty} b_{n,k} r_k$, such that $\| f_n \|_{M_{p, w}} = 1$ and
\begin{equation} \label{43}
\sum_{k=1}^{\infty} b_{n,k}^2 \rightarrow 0 ~~{\rm as} ~~ n \rightarrow \infty.
\end{equation}
Observe that $\{f_n\}_{n=1}^{\infty}$ does not contain any subsequence converging in $M_{p, w}$-norm. In fact, if 
$\| f_{n_k} - f \|_{M_{p, w}} \rightarrow 0$ for some $\{f_{n_k}\} \subset \{f_n\}$ and $f \in X$, then from Theorem \ref{Thm1} and 
(\ref{43}) it follows that $f =  \sum_{k=1}^{\infty} b_k r_k$, where $b_k = 0$ for all $k = 1, 2, \ldots$. Hence, $f = 0$. On the other 
hand, obviously, $\| f \|_{M_{p, w}} = 1$, and we come to a contradiction.

Thus, passing if it is needed to a subsequence, we can assume that
\begin{equation} \label{44}
\| f_n - f_m \|_{M_{p, w}} \geq \varepsilon > 0 ~~ {\rm for ~ all} ~~ n \neq m.
\end{equation}
Recall that, by Corollary \ref{Cor3}, the sequence $\{r_k\}_{k=1}^{\infty}$ is a basis of the space $({\cal R}_p)^*$. 
Applying the diagonal process, we can find the sequence $\{n_k\}_{k=1}^{\infty}, n_1 < n_2 < \ldots$, such that for every 
$i = 1, 2, \ldots$ there exists $\lim_{k \rightarrow \infty} \int_0^1r_i(s)f_{n_k}(s)\,ds$. Then,
$$
\lim_{k \rightarrow \infty} \int_0^1r_i(s)(f_{n_{2k+1}}(s) - f_{n_{2k}}(s))\,ds = 0 ~~ {\rm for ~ all} ~~ i = 1, 2, \ldots .
$$
Hence, since the sequence $\{f_{n_{2k+1}} - f_{n_{2k}}\}_{k=1}^{\infty}$ is bounded in $M_{p, w}$ we infer that 
$f_{n_{2k+1}} - f_{n_{2k}} \rightarrow 0$ weakly in $M_{p, w}$. Now, taking into account (\ref{44}) and applying the well-known 
Bessaga-Pe{\l}czy\'nski Selection Principle (cf. \cite[Proposition 1.3.10, p. 14]{AK06}), we may construct a subsequence of the 
sequence $\{f_{n_{2k+1}} - f_{n_{2k}}\}_{k=1}^{\infty}$ (we keep for it the same notation) and a block basis 
$$
u_k = \sum_{j=m_k}^{m_{k+1} -1} a_j r_j, ~ 1 = m_1 < m_2 < \ldots,
$$
such that 
\begin{equation} \label{45}
\| u_k - \big(f_{n_{2k+1}} - f_{n_{2k}} \big) \|_{M_{p, w}} \leq B_0^{-1} \cdot 2^{-k-1}, ~~ k = 1, 2, \ldots,
\end{equation}
where $B_0$ is the basis constant of $\{r_k\}$ in ${\cal R}_p$, and 
\begin{equation} \label{46}
w(2^{-m_{k+1}}) \leq \frac{1}{2} \cdot w(2^{-m_k}), ~~ k = 1, 2, \ldots .
\end{equation}
From (\ref{45}) it follows that the sequences $\{u_k\}_{k=1}^{\infty}$ and $\{f_{n_{2k+1}} - f_{n_{2k}}\}_{k=1}^{\infty}$ are equivalent 
in $M_{p, w}$ (cf. \cite[Proposition 1.a.9]{LT77}). Moreover, by Theorem \ref{Thm1} and \eqref{43},
$$
\sum_{j= m_k}^{m_{k+1} - 1} a_j^2 \rightarrow 0 ~~ {\rm as} ~~ k \rightarrow \infty.
$$
This fact together with inequality (\ref{46}) allows us to apply Proposition \ref{Pro5}, which implies that the sequence 
$\{u_k\}_{k=1}^{\infty}$ (and so $\{f_{n_{2k+1}} - f_{n_{2k}}\}_{k=1}^{\infty}$) contains a subsequence equivalent to the unit 
vector basis of $c_0$. Since $\{f_{n_{2k+1}} - f_{n_{2k}}\}_{k=1}^{\infty} \subset X$, then $X$ contains a subspace isomorphic to 
$c_0$. Complementability of this subspace in ${\cal R}_p$ is an immediate consequence of Sobczyk's theorem 
(see \cite[Corollary 2.5.9]{AK06}).
\endproof

\begin{remark}
{\rm If $\lim_{t \rightarrow 0^+} w(t) > 0$, then $M_{p, w} = L_{\infty}$ and $\{r_k\}$ is equivalent in $M_{p, w}$ 
to the unit vector basis of $l_1$ (cf. Theorem \ref{Thm1}). Observe also that if $\sup_{0 < t \leq 1} w(t) \log_2^{1/2} (2/t) < \infty$, 
then we get another trivial situation: ${\cal R}_p \simeq l_2$ (see Corollary \ref{2}).}
\end{remark}


\vspace{3mm}

\noindent
{\footnotesize Department of Mathematics and Mechanics, Samara State University\\
Acad. Pavlova 1, 443011 Samara, Russia} ~{\it E-mail address:} {\tt astash@samsu.ru} \\

\vspace{-3mm}

\noindent
{\footnotesize Department of Mathematics, Lule\r{a} University of Technology\\
SE-971 87 Lule\r{a}, Sweden} ~{\it E-mail address:} {\tt lech@sm.luth.se} \\


\begin{thebibliography}{99}

\vspace{-3mm}

\small{
\bibitem{AK06} F. Albiac and N. J. Kalton, {\it Topics in Banach Space Theory}, Graduate 
Texts in Mathematics 233, Springer, New York 2006.
\vspace{-2mm}

\bibitem{AA81} J. Alvarez Alonso, {\it The distribution function in the Morrey space}, Proc. Amer. 
Math. Soc. {\bf 83} (1981), no. 4, 693--699. 
\vspace{-2mm}

\bibitem{As01} S. V. Astashkin, {\it About interpolation of subspaces of rearrangement
invariant spaces generated by Rademacher system}, Int. J. Math. Sci. {\bf 25} (2001),
no. 7, 451--465.
\vspace{-2mm}

\bibitem{As09} S. V. Astashkin, {\it Rademacher functions in symmetric spaces},  Sovrem. Mat. Fundam. Napravl.  
{\bf 32} (2009), 3--161;  English transl. in J. Math. Sci. (N. Y.) {\bf 169} (2010), no. 6, 725--886. 
\vspace{-2mm}

\bibitem{ALM11} S. V. Astashkin, M. Leibov and L. Maligranda, {\it Rademacher functions in BMO}, Studia 
Math. {\bf 205} (2011), no. 1, 83--100.
\vspace{-2mm}

\bibitem{AM10} S. V. Astashkin and L. Maligranda, {\it Rademacher functions in Ces\`aro type spaces}, 
Studia Math. {\bf 198} (2010),  no. 3, 235--247.
\vspace{-2mm}

\bibitem{BP58S} C. Bessaga and A. Pe{\l}czy\'nski, {\it  On bases and unconditional convergence of series in Banach 
spaces}, Studia Math. {\bf 17} (1958), 151--164.
\vspace{-2mm}

\bibitem{BP58} C. Bessaga and A. Pe{\l}czy\'nski, {\it Some remarks on conjugate spaces containing subspaces 
isomorphic to the space $c_{0}$}, Bull. Acad. Polon. Sci. S\'er. Sci. Math. Astr. Phys. {\bf 6} (1958), 249--250.
\vspace{-6mm}

\bibitem{BRV99} O. Blasco, A. Ruiz, and L. Vega, {\it Non-interpolation in Morrey-Campanato and block 
spaces}, Ann. Scuola Norm. Sup. Pisa Cl. Sci. (4) {\bf 28} (1999), no. 1, 31--40. 
\vspace{-2mm}

\bibitem{BK91} Yu.A. Brudnyi and N.Ya. Krugljak, {\it Interpolation Functors and Interpolation Spaces}, North-Holland, 
Amsterdam 1991.
\vspace{-2mm}

\bibitem {DJT}J. Diestel, H. Jarchow and A. Tonge, {\it Absolutely Summing Operators}, Cambridge Univ. 
Press 1995.
\vspace{-2mm}

\bibitem {KKL} B. I. Korenblyum, S. G. Kre\u\i n and B. Ya. Levin, {\it On certain nonlinear 
questions of the theory of singular integrals}, Doklady Akad. Nauk SSSR (N.S.) {\bf 62} (1948), 
17--20 (Russian).
\vspace{-2mm}

\bibitem{KPS} S. G. Krein, Yu. I. Petunin, and E. M. Semenov, {\it Interpolation
of Linear Operators}, Nauka, Moscow, 1978 (Russian); English transl. in
Amer. Math. Soc., Providence 1982.
\vspace{-2mm}

\bibitem{KJF77} A. Kufner, O. John and S. Fu{\v c}ik, {\it Function Spaces}, Academia, 
Prague 1977.
\vspace{-2mm}

\bibitem{LR13} P. G. Lemari\'e-Rieusset, {\it Multipliers and Morrey spaces}, Potential Anal. 
{\bf 38} (2013), 741--752.
\vspace{-2mm}

\bibitem{LT77} J. Lindenstrauss and L. Tzafriri, {\it Classical Banach
Spaces, I. Sequence Spaces}, Springer-Verlag, Berlin-New York 1977.
\vspace{-2mm}

\bibitem{LT} J. Lindenstrauss and L. Tzafriri, {\it Classical Banach
Spaces, II. Function Spaces}, Springer-Verlag, Berlin-New York 1979.
\vspace{-2mm}

\bibitem{LZ66} W. A. J. Luxemburg and A. C. Zaanen, {\it Some examples of
normed K\"othe spaces}, Math. Ann. {\bf 162} (1966), 337--350.
\vspace{-2mm}

\bibitem{MP84} M. B. Marcus and G. Pisier, {\it Characterizations of almost surely
continuous $p$-stable random Fourier series and strongly stationary processes},
Acta Math. {\bf 152} (1984),  no. 3-4, 245--301.
\vspace{-2mm}

\bibitem{Mo38} C.B. Morrey, {\it On the solutions of quasi-linear elliptic partial differential equations}, Trans. Amer. 
Math. Soc. {\bf 43} (1938) 126--166.
\vspace{-2mm}

\bibitem{Na11} E. Nakai, {\it Orlicz-Morrey spaces and their preduals}, in: ``Banach and Function Spaces III", 
Proc. of the Third Internat. Symp. on Banach and Function Spaces (ISBFS2009) (14-17 Sept. 2009, 
Kitakyushu-Japan), Edited by M. Kato, L. Maligranda and T. Suzuki, Yokohama Publishers 2011, 173--186.
\vspace{-2mm}

\bibitem{PZ30} R. E. A. C. Paley and A. Zygmund, {\it On some series of functions. I, II}, 
Proc. Camb. Phil. Soc. {\bf 26} (1930), 337--357, 458--474.
\vspace{-2mm}

\bibitem{Pe69} J. Peetre, {\it On the theory of ${\cal L}_{p, \lambda}$ spaces}, J. Funct. Anal. {\bf 4} 
(1969), 71--87. 
\vspace{-2mm}

\bibitem{Pi81} G. Pisier, {\it De nouvelles caract\'erisations des ensembles de Sidon}
[{\it Some new characterizations of Sidon sets}], in: Mathematical Analysis and Applications,
Part B, Adv. in Math. Suppl. Stud., 7b, Academic Press, New York-London 1981, 685--726.
\vspace{-2mm}

\bibitem{RS75} V. A. Rodin and E. M. Semyonov, {\it Rademacher series in symmetric
spaces}, Anal. Math. {\bf 1} (1975), no. 3, 207--222.
\vspace{-2mm}

\bibitem{RS79} V. A. Rodin and E. M. Semenov, {\it The complementability of a subspace that is generated 
by the Rademacher system in a symmetric space}, Funktsional. Anal. i Prilozhen.
{\bf 13} (1979), no. 2, 91--92; English transl. in Functional Anal. Appl. {\bf 13} (1979), no. 2, 150--151.
\vspace{-2mm}

\bibitem {Ru91} W. Rudin, {\it Functional Analysis}, 2nd Edition, MacGraw-Hill, New York 1991.
\vspace{-2mm}

\bibitem{Za83} A. C. Zaanen, {\it Riesz Spaces II}, North-Holland, Amsterdam 1983.
\vspace{-2mm}

\bibitem{Zo86} C. T. Zorko, {\it Morrey spaces}, Proc. Amer. Math. Soc. {\bf 98} (1986), no. 4, 586--592.
\vspace{-2mm}

\bibitem{Zy35} A. Zygmund, {\it Trigonometrical Series}, Fundusz Kultury Narodowej, Warszawa-Lw\'ow 1935.
\vspace{-2mm}

\bibitem{Zy59} A. Zygmund, {\it Trigonometric Series. 2nd ed. Vols. I, II}, Cambridge University Press, 
New York 1959.

}
\end{thebibliography}
\end{document}